\documentclass[11pt]{article}
\usepackage{amsmath}
\usepackage{caption}
\usepackage{multicol}
\usepackage{geometry}
\usepackage{amssymb,amsthm}
\usepackage{amsmath}
\usepackage{amscd}
\usepackage{epsf,graphicx}
\usepackage{ytableau}
\usepackage[hidelinks]{hyperref}
\usepackage{tikz}
\usepackage{colortbl}
\usepackage{tikz-cd}
\usepackage[all]{xy}
\newtheorem{thm}{Theorem}
\newtheorem{prop}[thm]{Proposition} 
\newtheorem{rem}[thm]{Remark}

\newtheorem{defi}[thm]{Definition}
\newtheorem{exa}[thm]{Example}

\definecolor{amber}{rgb}{1.0, 0.75, 0.0} 
\definecolor{americanrose}{rgb}{1.0, 0.01, 0.24}
\definecolor{amber(sae/ece)}{rgb}{1.0, 0.49, 0.0}

\begin{document}
\setlength{\baselineskip}{16pt}
\title{RSK and Quantum Symmetric Functions: A Combinatorial Approach}
\author{Eddy Pariguan{\footnote{ Pontificia Universidad Javeriana. Bogotá. Colombia.
\texttt{epariguan@javeriana.edu.co} }} and   Jhoan  Sierra V.{\footnote{ Pontificia Universidad Javeriana. Bogotá.
 Colombia. \texttt{jhoansierra@javeriana.edu.co} }}}
\maketitle
\begin{abstract}
We explore an application of the Robinson–Schensted–Knuth (RSK) algorithm in the context of the quantum product of multi-symmetric functions. After reviewing the combinatorial foundations of quantum symmetric functions, we establish connections with transportation polytopes. Our approach highlights the combinatorial richness underlying the star product of symmetric functions.

\end{abstract}

\section*{Introduction}

Deformation quantization provides a powerful framework for understanding the transition from classical to quantum mechanics by deforming the commutative algebra of classical observables into a non‑commutative one. Rather than replacing functions with operators, deformation quantization retains the algebra of functions but modifies their product via a formal parameter \(\hbar\), resulting in the so‑called star product. A major breakthrough in this area was Kontsevich’s general solution to the deformation quantization problem for arbitrary Poisson manifolds \cite{Kon}.

In the foundational work on flat phase space, Moyal (and earlier Groenewold) provided an explicit formula for this star product—now known as the Moyal product—well before Kontsevich’s general construction \cite{Groenewold,Moyal}.
In parallel, the formalism of deformation quantization was systematically developed by Bayen, Flato, Fronsdal, Lichnerowicz, and Sternheimer, who established its physical motivation and mathematical foundations in the symplectic setting \cite{Bayen}. Moreover, alternative geometric approaches have since been proposed, such as those by de Wilde and Lecomte \cite{deWildeLecomte} and Fedosov \cite{Fedosov}.

The algebra of quantum symmetric functions arises naturally as the deformation quantization of the Poisson orbifold \(\mathbb{R}^{2n}/S_n\). In \cite{DP2}, Díaz and Pariguan provided explicit combinatorial formulae for the quantum product of multisymmetric functions, building on Vaccarino’s classical description of these functions via elementary generators \cite{Vac}. Classical multisymmetric functions were also studied earlier in the context of invariant theory and combinatorics, including contributions by Dalbec in the 1990s \cite{Dalbec}.

The main goal of this article is to provide a detailed combinatorial and computational approach to the structure of the star product of elementary multisymmetric functions. Our analysis is based on the decomposition of the quantum product into a sum over certain cubical matrices \(\Gamma \in Q(\alpha, \beta, n, m)\), which satisfy prescribed marginal conditions. We reinterpret these objects using  transportation polytopes and introduce the notions of \emph{support} and \emph{lifting}, which connect the quantum setting to classical matrices in \(L(\alpha, \beta, n)\).

One of the main contributions of this paper is the adaptation of a variant of the Robinson-Schensted-Knuth (RSK) algorithm tailored for application in the quantum case.To this end, we introduce a new class of combinatorial objects, called \emph{3‑words}, which encode cubical matrices in a form amenable to RSK‑type constructions. We show that these 3‑words provide an efficient way to reconstruct and compute the relevant matrices involved in the quantum product. This reformulation proves particularly suitable for implementation via Python-based algorithms.

We begin by recalling Vaccarino’s representation of the ring of multisymmetric functions, which serves as the foundation for the quantum constructions developed in \cite{DP2}. The elementary multisymmetric functions form a basis for the classical ring, and their combinatorial structure is encoded by matrices $\gamma$ satisfying marginal constraints. These matrices will later reappear, in deformed form, in the quantum setting.

Let $\mathbb{K}[x_{11},\cdots, x_{1d},\cdots,x_{n1},\cdots, x_{nd} ]^{\mathbb{S}_n}$ denote the ring of multisymmetric functions, that is, the subring of $\mathbb{K}[x_{ij}]$ invariant under the diagonal action of the symmetric group $\mathbb{S}_n$ on the $n$ copies of $\mathbb{K}^d$. 

Given $p = (p_1, \ldots, p_a) \in \mathbb{K}[y_1,\ldots,y_d]^a$ and a multi-index $\alpha \in \mathbb{N}^a$ such that $|\alpha| \leq n$, the elementary multisymmetric function $e_{\alpha}(p)$ is defined by the generating identity:
\[
\prod_{i=1}^{n} \left(1 + p_1(i)t_1 + \cdots + p_a(i)t_a \right) 
= \sum_{|\alpha|\leq n} e_{\alpha}(p)\, t^\alpha,
\]
where to compute $ p_j(i)$, each variable $ y_i $ in $p_j $ is replaced with the corresponding variable $ x_{ij}$ for  $j\in[d]$.

The set of functions $\{e_\alpha(p)\}$ forms a basis of the ring of multisymmetric functions. Moreover, their classical product satisfies the following formula.

\begin{thm}\label{classicalProductVaccarino}
Let $\alpha \in \mathbb{N}^a$ and $\beta \in\mathbb{N}^b$ with $|\alpha|, |\beta| \leq n$, and let $p = (p_1,\ldots,p_a)$, $q = (q_1,\ldots,q_b)$ be polynomials in $\mathbb{K}[y_1,\ldots,y_d]$. Then:
\[
e_{\alpha}(p) \cdot e_{\beta}(q) = \sum_{\gamma \in L(\alpha, \beta, n)} e_{\gamma}(p, q, pq),
\]
where $(p,q,pq)$ denotes the concatenation of the lists $p$, $q$, and all pairwise products $p_i q_j$, and where $L(\alpha, \beta, n)$ is the set of matrices $\gamma$ with entries in $\mathbb{N}$ indexed by $(\{0\} \cup [a]) \times (\{0\} \cup [b])$ such that:
\begin{itemize}
    \item $\gamma_{00} = 0$,
    \item $\sum_{i=0}^a \sum_{j=0}^b \gamma_{ij} \leq n$,
    \item $\sum_{j=0}^b \gamma_{ij} = \alpha_i$ for $i \in [a]$,
    \item $\sum_{i=0}^a \gamma_{ij} = \beta_j$ for $j \in [b]$.
\end{itemize}
\end{thm}

The matrices $\gamma \in L(\alpha, \beta, n)$ encode the combinatorics of how monomials in the variables $p_i$, $q_j$, and $p_i q_j$ combine in the product. They are naturally represented as follows:

\[
\begin{array}{cc}
	\begin{array}{ c c c c}
		\phantom{abc} & \beta_1 & \cdots & \beta_b \\
		\mbox{} & \uparrow & \mbox{}  & \uparrow\\
	\end{array} &  
	\begin{array}{cc}
		\mbox{}  & \mbox{} \\
		\mbox{}  & \mbox{} \\
	\end{array}\\
	\begin{bmatrix}
		0          &\gamma_{01}&\cdots&\gamma_{0b}\\
		\gamma_{10}&\gamma_{11}&\cdots&\gamma_{1b}\\
		\vdots & \vdots & \ddots & \vdots\\
		\gamma_{a0}&\gamma_{a1}&\cdots&\gamma_{ab}
	\end{bmatrix} & 
	\begin{array}{cc}
		\mbox{} & \mbox{}\\
		\to & \alpha_1\\
		\mbox{}  & \vdots\\
		\to & \alpha_a
	\end{array}
\end{array}
\]

This representation will play a key role in the deformation process developed in later sections, where matrices $\gamma$ will be generalized to cubical matrices $\Gamma$ parametrizing quantum products.

\section{Quantum product and a combinatorial example}

After presenting the classical product of multisymmetric functions, we now turn to its quantum deformation, introduced in \cite{DP2}, which extends Vaccarino’s formula using Kontsevich’s star product on the phase space $\mathbb{R}^2$. The main result gives an explicit combinatorial formula involving cubical matrices $\Gamma$, generalizing the classical matrices $\gamma$. 

To illustrate the construction, we include a computational example of the quantum star product, highlighting the role of the deformation parameter $\hbar$ and the algorithmic tools involved.

It is well known that the Weyl algebra is isomorphic to the deformation quantization of the algebra of polynomial functions on $\mathbb{R}^2$ endowed with the canonical Poisson structure. Consequently, the deformation quantization of the space $(\mathbb{R}^2)^n/S_n$, where $S_n$ denotes the symmetric group, can be identified with the algebra of quantum symmetric functions
$$(\mathbb{R}[x_1,\cdots,x_n,y_1,\cdots,y_n][[\hbar]],\star)^{S_n},$$
or, equivalently, with the symmetric powers of the Weyl algebra, $(W^{\otimes n})/S_n$.

We now state the main result from \cite{DP2}, which provides an explicit formula for the star product of multisymmetric functions. While we do not reproduce the proof here, we use this theorem as the starting point for the present work. Our focus will be on a specific family of matrices that naturally arise in the structure of the formula, and which play a central role in the analysis developed throughout this paper.

\begin{thm}\label{TeoImportante}

    Consider $\mathbb{R}^{2}$ with its canonical Poisson structure. Fix $a, b\in\mathbb{N}^{+}$ and let $$
    p=(p_1,\cdots,p_a)=
    (x^{c_1}y^{d_1},\cdots,x^{c_a}y^{d_a})\in \mathbb{R}[x,y]^{a},$$ $$
    q=(q_1,\cdots,q_b)=(x^{f_1}y^{g_1},\cdots,x^{f_b}y^{g_b} ) \in \mathbb{R}[x, y]^b
    .$$
For $\alpha \in \mathbb{N}^a$ and $\beta \in \mathbb{N}^b$, the following identity holds:
$$
e_\alpha(p) \star e_\beta(q) = \sum_{m=0}^{\infty} \left( \sum_{\Gamma \in Q(\alpha,\beta,n,m)} e_\Gamma(B(p,q)) \right) \hbar^m,
$$
where:
\begin{itemize}
    \item $B(p,q) = (p, q, p_1 \star q_1, \cdots, p_l \star q_r, \cdots, p_a \star q_b)$, and
    $$
    p_l \star q_r = x^{c_l}y^{d_r} \star x^{f_l}y^{g_r} =
    \sum_{k=0}^{\min(d_r, f_l)} \binom{d_r}{k}(f_l)_{k} x^{c_l + f_l - k} y^{d_r + g_r - k} \hbar^k.
    $$
\item $Q(\alpha,\beta,n,m)$ is the subset of  $Map([0,a]\times[0,b]\times\mathbb{N},\mathbb{N})$ consisting of cubical matrices  $\Gamma:[0,a]\times[0,b]\times\mathbb{N}\longrightarrow\mathbb{N}$ such that
        	 \begin{enumerate}
        \item $\Gamma_{00}^{k} = 0$ for all $k \geq 0$; and if either $i = 0$ or $j = 0$, then $\Gamma_{ij}^{k} = 0$ for all $k \geq 1$,
        \item $|\Gamma| = \sum_{i=0}^{a} \sum_{j=0}^{b} \sum_{k=0}^{\infty} \Gamma_{ij}^{k} \leq n$, and $\sum_{i=0}^{a} \sum_{j=0}^{b} \sum_{k=0}^{\infty} k \Gamma_{ij}^{k} = m$,
        \item $\sum_{j=0}^{b} \sum_{k=0}^{\infty} \Gamma_{ij}^{k} = \alpha_i$ for all $i \in [a]$, and $\sum_{i=0}^{a} \sum_{k=0}^{\infty} \Gamma_{ij}^{k} = \beta_j$ for all $j \in [b]$.
    \end{enumerate}
	\end{itemize}
\end{thm}

To illustrate the application of Theorem \ref{TeoImportante}, we now present a comprehensive computation of a star product between two elementary multisymmetric functions. This process involves solving a system of algebraic constraints that determine the admissible matrices $\Gamma$, making the calculation too intricate to perform manually. To address this, we developed a Python implementation of the method. The corresponding code is available in our GitHub repository~\cite{sierra2025}, enabling readers to reproduce the results and experiment with additional cases.

We consider the specific case $\alpha = (1,1)$, $\beta = (2,1)$, with inputs $p = (x^2y, x^3y)$ and $q = (x^3, x^2y^2)$. Setting $n = 4$, we compute the quantum product
    $$e_\alpha(p)\star e_\beta(q)=\sum_{m=0}^{\infty}\left(\sum_{\Gamma\in Q(\alpha,\beta,n,m)} e_\Gamma(B(p,q))\right)\hbar^m.$$
    First we find $B(p,q)$, which is given by
$$B(p,q)=({x^2y}, {x^3y}, {x^3}, {x^2y^2}, {x^5y}, {x^4}, {x^4y^3}, {x^3y^2}, {x^6y}, {x^5}, {x^5y^3}, {x^4y^2}).$$

Next, we determine the set of cubical matrices $\Gamma \in Q(\alpha,\beta,n,m)$ that satisfy the required conditions. Each such $\Gamma$ consists of a infinite collection of matrices indexed by $k \in \mathbb{N}$, and must be interpreted as a three-dimensional array whose entries $\Gamma_{ij}^k$. For each value of $m$, we identify all such matrices $\Gamma$ and describe their structure both in cubical form and via their corresponding vector representations.

As an example, let $m = 0$. One such matrix is:

    $$\Gamma = \left[
        \begin{bmatrix}
            0 & \cellcolor{green}{1} & \cellcolor{green}{0}\\
            \cellcolor{yellow}{0} & \cellcolor{green}{1} & \cellcolor{green}{0}\\
            \cellcolor{yellow}{0} & \cellcolor{green}{0} & \cellcolor{green}{1}
        \end{bmatrix}\right]= \left[
        \begin{bmatrix}
            0 & \cellcolor{green}{1} & \cellcolor{green}{0}\\
            \cellcolor{yellow}{0} & \cellcolor{green}{1} & \cellcolor{green}{0}\\
            \cellcolor{yellow}{0} & \cellcolor{green}{0} & \cellcolor{green}{1}
        \end{bmatrix},
        \begin{bmatrix}
            0 & 0 & 0\\
            0 & \cellcolor{green}{0} & \cellcolor{green}{0}\\
            0 & \cellcolor{green}{0} & \cellcolor{green}{0}
        \end{bmatrix}, \cdots
    \right],$$
    and its corresponding vector representation is
    $$(\colorbox{yellow}{0, 0},\colorbox{green}{1, 0, 1, 0, 0, 1, 0, 0, 0, 0}).$$

 Note that, although we are working with matrices whose entries $\Gamma_{ij}^{k}$ range over all $k\geq 0$, we only consider terms up to the highest index $k$ for which at least one entry is nonzero. 
Now let $m = 1$. One such matrix is:
    $$\Gamma = \left[
        \begin{bmatrix}
            0 & \cellcolor{green}{1} & \cellcolor{green}{0}\\
            \cellcolor{yellow}{0} & \cellcolor{green}{1} & \cellcolor{green}{0}\\
            \cellcolor{yellow}{0} & \cellcolor{green}{0} & \cellcolor{green}{0}
        \end{bmatrix},
        \begin{bmatrix}
           0 & 0 & 0\\
           0 & \cellcolor{green}{0} & \cellcolor{green}{0}\\
           0 & \cellcolor{green}{0} & \cellcolor{green}{1} 
        \end{bmatrix}\right],
    $$ 
 and its corresponding vector representation is
    $$(\colorbox{yellow}{0, 0},\colorbox{green}{1, 0, 1, 0, 0, 0, 0, 0, 0, 1}).$$

 In order to represent each cubical matrix $\Gamma$ as a vector, we adopt the following convention. First, the yellow column is transposed and placed at the beginning of the vector.Then the remaining entries in the first matrix (highlighted in green) are read row by row, from top to bottom and right to left, and concatenated. The same procedure is applied to the green entries of each subsequent matrix in $\Gamma$. This yields a finite vector of integers that encodes the structure of $\Gamma$ in a compact and computationally convenient format.

We now compute the full expansion by varying $m$ from $0$ to $\infty$, as prescribed in the formula. Each value of $m$ corresponds to the degree of the deformation parameter $\hbar^m$ in the expansion. However, we observe that for $ m \geq 3$, although the sets $ Q(\alpha, \beta, n, m)$ may contain elements, none of them contribute nonzero terms to the star product. Therefore, only the cases $m = 0, 1, 2 $ need to be considered. The vector solutions for each of these values of $m$ are listed below, grouped by degree in $\hbar$.
    \begin{itemize}
     \item \textbf{Vectors associated with degree $\hbar^0$ \ (i.e., $ m = 0$)}:           
                
            $$ \begin{array}{cc}
                (0, 0, 1, 0, 1, 0, 0, 1, 0, 0, 0, 0),&
                (0, 0, 1, 0, 0, 1, 1, 0, 0, 0, 0, 0),\\
                (0, 0, 0, 1, 1, 0, 1, 0, 0, 0, 0, 0),&
                (1, 0, 2, 0, 0, 0, 0, 1, 0, 0, 0, 0),\\
                (0, 1, 2, 0, 0, 1, 0, 0, 0, 0, 0, 0),&
                (1, 0, 1, 1, 0, 0, 1, 0, 0, 0, 0, 0),\\
                (0, 1, 1, 1, 1, 0, 0, 0, 0, 0, 0, 0).\\
            \end{array}$$
 \item \textbf{Vectors associated with degree $\hbar$ \ (i.e., $ m = 1$)}:           
            $$\begin{array}{cc}
                 (0, 0, 1, 0, 1, 0, 0, 0, 0, 0, 0, 1),&
                 (0, 0, 1, 0, 0, 1, 0, 0, 0, 0, 1, 0),\\
                 (0, 0, 0, 1, 1, 0, 0, 0, 0, 0, 1, 0),&
                 (0, 0, 1, 0, 0, 0, 1, 0, 0, 1, 0, 0),\\
                 (0, 0, 1, 0, 0, 0, 0, 1, 1, 0, 0, 0),&
                 (0, 0, 0, 1, 0, 0, 1, 0, 1, 0, 0, 0),\\
                 (1, 0, 2, 0, 0, 0, 0, 0, 0, 0, 0, 1),&
                 (1, 0, 1, 1, 0, 0, 0, 0, 0, 0, 1, 0),\\
                 (0, 1, 1, 1, 0, 0, 0, 0, 1, 0, 0, 0).
            \end{array}$$
            
                 \item \textbf{Vectors associated with degree $\hbar^2$ \ (i.e., $ m = 2$)}:                       
     $$\begin{array}{cc}
                (0, 0, 1, 0, 0, 0, 0, 0, 1, 0, 0, 1),&
                (0, 0, 1, 0, 0, 0, 0, 0, 0, 1, 1, 0),\\
                (0, 0, 0, 1, 0, 0, 0, 0, 1, 0, 1, 0).
            \end{array}$$
    \end{itemize}
Each of these vectors contributes an elementary multisymmetric function to the star product. To determine the specific function associated with a given vector, we proceed as follows.

Take, for instance, the vector
$(0, 0, 1, 0, 1, 0, 0, 1, 0, 0, 0, 0).$
To identify the corresponding elementary function, we refer to the previously constructed vector $B(p,q)$ and eliminate both the zero entries in the solution vector and their associated positions in $B(p,q)$. That is, we retain only the variables corresponding to the nonzero positions. This yields the following term:    
    $$\begin{array}{cc}
         & e_{(\textcolor{red}{0}, \textcolor{red}{0}, 1, \textcolor{red}{0}, 1, \textcolor{red}{0}, \textcolor{red}{0}, 1, \textcolor{red}{0}, \textcolor{red}{0}, \textcolor{red}{0}, \textcolor{red}{0})}(\textcolor{red}{x^2y},
         \textcolor{red}{x^3y},
         {x^3},
         \textcolor{red}{x^2y^2},
         {x^5y},
         \textcolor{red}{x^4},
         \textcolor{red}{x^4y^3},
         {x^3y^2},
         \textcolor{red}{x^6y},
         \textcolor{red}{x^5},
         \textcolor{red}{x^5y^3},
         \textcolor{red}{x^4y^2}) \\
         & = e_{(1,1,1)}(x^{3},x^{5}y,x^{3}y^{2}).
    \end{array}$$

This procedure applies to all solution vectors, systematically generating the elementary multisymmetric terms that contribute to the quantum product at each order of $\hbar$.

$$\begin{array}{rcl}
    e_{(1,1)}(x^2y,x^3y) \star e_{(2,1)}(x^3,x^2y^2)
    & = & e_{(1,1,1)}(x^{3},x^{5}y,x^{3}y^{2})\\
    + e_{(1,1,1)}(x^{3},x^{4},x^{4}y^{3}) &
    + & e_{(1,1,1)}(x^{2}y^{2},x^{5}y,x^{4}y^{3})\\
    + e_{(1,2,1)}(x^{2}y,x^{3},x^{3}y^{2})
    & + & e_{(1,2,1)}(x^{3}y,x^{3},x^{4})\\
    + e_{(1,1,1,1)}(x^{2}y,x^{3},x^{2}y^{2},x^{4}y^{3})
    & + & e_{(1,1,1,1)}(x^{3}y,x^{3},x^{2}y^{2},x^{5}y)\\
    + e_{(1,1,1)}(x^{3},x^{5}y,x^{4}y^{2})\hbar
    & + & e_{(1,1,1)}(x^{3},x^{4},x^{5}y^{3})\hbar\\
    + e_{(1,1,1)}(x^{2}y^{2},x^{5}y,x^{5}y^{3})\hbar
    & + & e_{(1,1,1)}(x^{3},x^{4}y^{3},x^{5})\hbar\\
    + e_{(1,1,1)}(x^{3},x^{3}y^{2},x^{6}y)\hbar
    & + & e_{(1,1,1)}(x^{2}y^{2},x^{4}y^{3},x^{6}y)\hbar\\
    + e_{(1,2,1)}(x^{2}y,x^{3},x^{4}y^{2})\hbar
    & + & e_{(1,1,1,1)}(x^{2}y,x^{3},x^{2}y^{2},x^{5}y^{3})\hbar\\
    + e_{(1,2,1)}(x^{3}y,x^{3},x^{5})\hbar
    & + & e_{(1,1,1,1)}(x^{3}y,x^{3},x^{2}y^{2},x^{6}y)\hbar\\
    + e_{(1,1,1)}(x^{3},x^{6}y,x^{4}y^{2})\hbar^{2}
    & + & e_{(1,1,1)}(x^{3},x^{5},x^{5}y^{3})\hbar^{2}\\
    + e_{(1,1,1)}(x^{2}y^{2},x^{6}y,x^{5}y^{3})\hbar^{2}.
    \end{array}$$

The set \( Q(\alpha, \beta, n, m) \) encodes the combinatorial structure underlying the star product expansion, particularly through the matrices \( \Gamma \) that determine the contributions of mixed monomials. To better understand the quantum deformation of the classical symmetric product, we now turn to a detailed combinatorial analysis of these matrices.

\section{From transportation polytopes to quantum products}
The set $ Q(\alpha, \beta, n, m)$ can be understood as the collection of integer matrices satisfying certain marginal conditions, reminiscent of the defining constraints of classical transportation polytopes \cite{Loera}, whose combinatorics provide a natural framework for organizing and interpreting the solutions.

Guided by this intuition, we aim to decompose the set $ Q(\alpha, \beta, n, m)$ into more manageable subsets, such that analyzing their individual structure provides insight into the global configuration. This motivates the introduction of the notion of \emph{support} for a matrix $ \Gamma \in Q(\alpha, \beta, n, m) $, inspired by the analogous concept in the classical setting.

\begin{defi}\label{defiSoporte}
Let \( s \in \mathbb{N}\). The \emph{support} of \( Q(\alpha, \beta, n, m) \) at level \( s \), denoted by \( Q^{s}(\alpha, \beta, n, m) \), is the subset of matrices \( \Gamma = (\Gamma_{ij}^k) \in Q(\alpha, \beta, n, m) \) such that:
\begin{itemize}
    \item \( \Gamma_{ij}^k = 0 \) for all \( k > s \) and for all \( i, j \),
    \item and there exists at least one pair \( (i,j) \) such that \( \Gamma_{ij}^s \neq 0 \).
\end{itemize}
\end{defi}

\begin{defi}
Let $\Gamma = (\Gamma_{ij}^k)$ be a cubical matrix in $Q(\alpha, \beta, n, m)$. For each $ k \in  \mathbb{N}$ , the \emph{$k$-th level} of $\Gamma$, denoted $ \Gamma^k $, is the matrix $ (\Gamma_{ij}^k)$ formed by fixing the index $k$  and varying $i, j$.
\end{defi}

Definition \ref{defiSoporte} can be understood as selecting those $\Gamma \in Q(\alpha,\beta,n,m)$ for which $\Gamma^k = 0$ for every level $k > s$, as depicted in Figure \ref{supposrStar}.

 \begin{center}
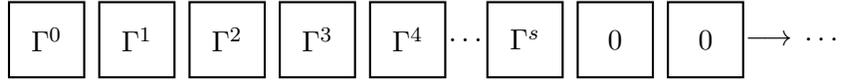

	\begin{tikzpicture}
		\def\startX{0}
		\def\startY{0}
		\def\size{1} 
		\def\gap{0.2} 
		
		\foreach \i in {0,1,2,3,4} {
			\draw[thick] (\startX+\i*\size+\i*\gap,\startY) rectangle ++(\size,\size);
			\node at (\startX+\i*\size+\i*\gap+\size/2,\startY+\size/2) {$\Gamma^{\i}$};
		}
		
		\node at (\startX+5.2*\size+4*\gap+\gap/2,\startY+\size/2) {\dots};
		
		\draw[thick] (\startX+5.3*\size+5.3*\gap,\startY) rectangle ++(\size,\size);
		\node at (\startX+5.3*\size+5.3*\gap+\size/2,\startY+\size/2) {$\Gamma^{s}$};
		
		\draw[thick] (\startX+6.3*\size+6.3*\gap,\startY) rectangle ++(\size,\size);
		\node at (\startX+6.3*\size+6.3*\gap+\size/2,\startY+\size/2) {$0$};
		
		\draw[thick] (\startX+7.3*\size+7.3*\gap,\startY) rectangle ++(\size,\size);
		\node at (\startX+7.3*\size+7.3*\gap+\size/2,\startY+\size/2) {$0$};

		\node at (\startX+8*\size+8*\gap+\size/2,\startY+\size/2) {$\longrightarrow$};
		
		\node[anchor=west] at (\startX+8.7*\size+8.7*\gap,\startY+\size/2) {$\cdots$};
	\end{tikzpicture}
    \captionof{figure}{Representation of the support $Q^{s}$ of $Q(\alpha,\beta,n,m)$ at level $s$.}\label{supposrStar}
\end{center}

\begin{exa}
Let $\alpha = (1,1)$, $\beta = (2,1)$, and $n = 4$. Consider  $\Gamma \in Q(\alpha,\beta,n,0)$ given by
\[
\Gamma = \left[
    \begin{bmatrix}
        0 & 1 & 0\\
        0 & 1 & 0\\
        0 & 0 & 1
    \end{bmatrix}\right]= \left[
    \begin{bmatrix}
        0 & 1 & 0\\
        0 & 1 & 0\\
        0 & 0 & 1
    \end{bmatrix},
    \begin{bmatrix}
        0 & 0 & 0\\
        0 & 0 & 0\\
        0 & 0 & 0
    \end{bmatrix}, \cdots
    \right].
\]
Observe that $\Gamma \in Q^0$. On the other hand, consider the $\Gamma \in Q(\alpha,\beta,n,1)$ defined as
\[
\Gamma = \left[
    \begin{bmatrix}
        0 & 1 & 0\\
        0 & 1 & 0\\
        0 & 0 & 0
    \end{bmatrix},
    \begin{bmatrix}
       0 & 0 & 0\\
       0 & 0 & 0\\
       0 & 0 & 1 
    \end{bmatrix}\right],
\]
then we have $\Gamma \in Q^1$.
\end{exa}

\begin{prop}\label{UDisjunta}
Let $\alpha\in\mathbb{N}^a$, $\beta \in\mathbb{N}^b$,$n,m$ be fixed. Then
$$Q(\alpha,\beta,n,m)=\bigsqcup_{s=0}^m Q^{s}.$$
\end{prop}

\begin{proof}
We distinguish two cases.

If $m = 0$, the condition
$$\displaystyle\sum_{i=0}^a\sum_{j=0}^b\sum_{k=1}^{\infty}k\Gamma_{ij}^k = 0$$
implies that $\Gamma_{ij}^k = 0$ for all $k > 0$, so $Q(\alpha,\beta,n,0) = Q^0$.

If $m > 0$, disjointness of the $Q^s$ follows from the definition: if $\Gamma \in Q^s \cap Q^{s'}$ with $s < s'$, then $\Gamma_{ij}^{s'}$ must be both zero and nonzero, a contradiction.

For $s > m$, any $\Gamma \in Q^s$ would have
$$k\Gamma_{ij}^k \leq m \Rightarrow \Gamma_{ij}^k < 1,$$
hence $\Gamma_{ij}^k = 0$, so $Q^s = \emptyset$.

Finally, for any $\Gamma \in Q(\alpha,\beta,n,m)$, we define
$$s = \max\{k \in [m] : \Gamma_{ij}^k \ne 0\}$$
and obtain $\Gamma \in Q^s$. This proves the decomposition.
\end{proof}

By Theorem~\ref{UDisjunta}, it is enough to study each set \( Q^{s} \) separately.

Throughout this section, \( \ell(v) \) denotes the number of entries in the vector \( v \). For instance, \( \ell(1,2,3) = 3 \).

Recall that in Theorem~\ref{TeoImportante}, the expression \( B(p,q) \) takes the form
$$B(p,q)=(p,q,p_1\star q_1,p_1\star q_2,\cdots,p_1\star q_b,\cdots,p_a\star q_b),$$
where each product \( p_i \star q_j \) is computed as follows:
\begin{equation}\label{starMonoid}
\begin{array}{ccc}
     (x^{c}y^{d})\star(x^{f}y^{g}) & = & \displaystyle\sum_{k=0}^{\min(d,f)} B_{k}(x^cy^d,x^fy^g)\hbar^k \\
      & = & \displaystyle\sum_{k=0}^{\min(d,f)}\binom{d}{k}(f)_{k}x^{c+f-k}y^{d+g-k}\hbar^k.
\end{array}
\end{equation}

\begin{prop}\label{ellBpq}
Let $p=(p_1,p_2,\cdots,p_{a})$ and $q=(q_1,q_2,\cdots,q_{b})$. Then
$$\ell(B(p, q))= a + b + \sum_{i=1}^{a} \sum_{j=1}^{b} \left( \min (\deg_y(p_i), \deg_x(q_j)) + 1 \right),$$
where $\deg_x(p_i)$ denotes the degree in the variable $x$ of the monomial $p_i$, and similarly, $\deg_y(q_j)$ is the degree in the variable $y$ of the monomial $q_j$.
\end{prop}

\begin{proof}
By definition, we have
\[
B(p,q) = (p, q, p_1 \star q_1, \ldots, p_1 \star q_b, \ldots, p_a \star q_b),
\]
so
\begin{equation}\label{longEquation}
\ell(B(p,q)) = \ell(p) + \ell(q) + \sum_{i=1}^{a} \sum_{j=1}^{b} \ell(p_i \star q_j).
\end{equation}

Using Equation~(\ref{starMonoid}), each product \( p_i \star q_j \) expands as a sum of \( \min(\deg_y(p_i), \deg_x(q_j)) + 1 \) terms, hence
\[
\ell(p_i \star q_j) = \min(\deg_y(p_i), \deg_x(q_j)) + 1.
\]

Substituting into Equation~(\ref{longEquation}) completes the proof.
\end{proof}

This naturally leads to the following question: what is the largest support $Q^{S}$ such that every element has all its nonzero entries within the first $\ell(B(p,q))$ positions? Here, we emphasize that we are referring to the index $S$ that defines the support set $Q^S$, not the support set itself.

\begin{prop}\label{Mmayor}
Let $p = (p_1,\ldots,p_a)$ and $q = (q_1,\ldots,q_b)$. The largest support reached in the star product $\star $ of two multisymmetric functions is
$$
S = \left\lceil \frac{\ell(B(p,q)) - (a + b)}{ab} \right\rceil - 1,
$$
where $ \lceil\cdot\rceil $ denotes the ceiling function.
\end{prop}

\begin{proof}
We consider an element \(\Gamma\) whose vector representation has length \(\ell(B(p,q))\). The first \(a\) entries correspond to the zero column (\(j = 0\)), and the next \(b\) entries to the zero row (\(i = 0\)) of level zero. Thus, the first \(a + b\) entries are entirely contained in level 0.

The remaining entries, totaling \(\ell(B(p,q)) - (a + b)\), must be distributed into blocks of size \(ab\), corresponding to the matrices of levels \(\Gamma^1, \Gamma^2, \dots\). To fit all these entries, we require
\[
\left\lceil \frac{\ell(B(p,q)) - (a + b)}{ab} \right\rceil
\]
such blocks. Since the levels are numbered starting from zero, the index of the highest level needed is
\[
S = \left\lceil \frac{\ell(B(p,q)) - (a + b)}{ab} \right\rceil - 1,
\]
which completes the proof.
\end{proof}

This significantly simplifies the computation by identifying the largest level $S$ one needs to analyze in order to obtain a term in the product $e_{\alpha}(p) \star e_{\beta}(q)$.
While Theorem~\ref{UDisjunta} describes the decomposition of $Q(\alpha,\beta,n,m)$ into disjoint levels, Proposition~\ref{Mmayor} determines an explicit bound on the level s beyond which no element $\Gamma \in Q^s$ contributes to the star product.
In this sense, Theorem~\ref{SMayor} provides a substantial improvement: it restricts the computation to a finite union of supports, making the explicit calculation of the star product more efficient and conceptually clearer.

\begin{thm}\label{SMayor}
Let $S$ be the largest support defined in Proposition~\ref{Mmayor}. Then
$$
e_{\alpha}(p)\star e_{\beta}(q)=\sum_{m=0}^{\infty} \sum_{\Gamma} e_\Gamma(B(p,q)) \hbar^m, \quad \text{where } \Gamma \text{ ranges over } \bigsqcup_{s=0}^{S} Q^{s}.
$$
\end{thm}

\begin{proof}
According to Proposition~\ref{Mmayor}, the star product of two multisymmetric functions only involves elements $\Gamma \in \bigsqcup_{s=0}^{S} Q^{s} $ . If $ \Gamma \in Q^{s} $ with $ s > S $, then its vector representation contains an entry $ \Gamma_{ij}^{s}$ that lies beyond the range of $B(p,q)$ and therefore contributes nothing to the expression. The sum can thus be restricted to levels $ s \leq S$, as claimed.
\end{proof}

So far, we have characterized the elements $\Gamma \in Q(\alpha,\beta,n,m) $. A natural question is whether there exists a meaningful relation between such elements $ \Gamma $ and the matrices $ \gamma \in L(\alpha,\beta,n) $ that appear in the classical product of multisymmetric functions. One approach to explore this connection is to consider the total matrix obtained by summing all levels $ \Gamma^k $ of a given $ \Gamma $, level by level.

\begin{exa}
	Let \( \Gamma \in Q(\alpha,\beta,n,1) \) with \( \alpha=(1,1) \), \( \beta=(2,1) \), and \( n=4 \). Let
	$$
	\Gamma =
	\left(
	\begin{array}{cc}
	\begin{bmatrix}
		0 & 1 & 0\\
		0 & 1 & 0\\
		0 & 0 & 0
	\end{bmatrix} &
	\begin{bmatrix}
		0 & 0 & 0\\
		0 & 0 & 0\\
		0 & 0 & 1
	\end{bmatrix} \\
	\Gamma^0 & \Gamma^1
	\end{array}
	\right).
	$$
	The matrix obtained by summing its levels is
	$$
	\overline{\Gamma} =
	\begin{bmatrix}
		0 & 1 & 0\\
		0 & 1 & 0\\
		0 & 0 & 1
	\end{bmatrix}.
	$$
\end{exa}
This leads us to define a new object, which we call the smash of a cubical matrix.
\begin{defi}\label{smash}
    Let $\Gamma \in Q(\alpha, \beta, n, m)$. The \textbf{smash} of $\Gamma$, denoted by $\overline{\Gamma}$, is the matrix $\overline{\Gamma} \in \mathrm{Mat}([0,a] \times [0,b], \mathbb{N}^{+})$ defined by
    \[
        \overline{\Gamma}_{ij} = \sum_{k=0}^{\infty} \Gamma_{ij}^k.
    \]

    More generally, given a subset $A \subseteq Q(\alpha,\beta,n,m)$, we define its \textbf{smash set} $\overline{A}$ as
    \[
        \overline{A} = \{ \overline{\Gamma} : \Gamma \in A \}.
    \]
\end{defi}

\begin{prop}\label{DeQaL}
    Let $\alpha  \in \mathbb{N}^a, \beta \in \mathbb{N}^b$, $n \in \mathbb{N}^{+}$ be fixed. Then, for every $m \in  \mathbb{N}^{+}$, the smash set of $Q(\alpha,\beta,n,m)$ is equal to $L(\alpha,\beta,n)$:
    $$
    \overline{Q(\alpha,\beta,n,m)} = L(\alpha,\beta,n).
    $$
\end{prop}

\begin{proof}
Let $\Gamma \in Q(\alpha,\beta,n,m)$ and define $\gamma = \overline{\Gamma}$ by $\gamma_{ij} = \sum_{k=0}^{\infty} \Gamma_{ij}^k$. We verify that $\gamma \in L(\alpha,\beta,n)$:
\begin{itemize}
    \item $\gamma_{00} = \sum_k \Gamma_{00}^k = 0$ since $\Gamma_{00}^k = 0$ for all $k$.
    \item $|\gamma| = |\Gamma| \leq n$.
    \item For all $i \in [a]$, $\sum_j \gamma_{ij} = \sum_j \sum_k \Gamma_{ij}^k = \alpha_i$.
    \item For all $j \in [b]$, $\sum_i \gamma_{ij} = \sum_i \sum_k \Gamma_{ij}^k = \beta_j$.
\end{itemize}
Thus, $\gamma \in L(\alpha,\beta,n)$ and $\overline{Q(\alpha,\beta,n,m)} \subseteq L(\alpha,\beta,n)$.

For the reverse inclusion, take any $\gamma \in L(\alpha,\beta,n)$. Let $(r,s)$ be such that $\gamma_{rs} \neq 0$ and define $\Gamma = (\Gamma_{ij}^k)$ by:
\begin{itemize}
    \item $\Gamma_{ij}^0 = \gamma_{ij}$ for $(i,j) \neq (r,s)$,
    \item $\Gamma_{rs}^0 = \gamma_{rs} - 1$, \quad $\Gamma_{rs}^m = 1$,
    \item $\Gamma_{ij}^k = 0$ otherwise.
\end{itemize}
Clearly, $\Gamma \in Q(\alpha,\beta,n,m)$ and $\overline{\Gamma} = \gamma$, completing the proof.
\end{proof}

The previous result establishes a precise correspondence between cubical matrices in $Q(\alpha, \beta, n, m)$ and matrices in $L(\alpha, \beta, n)$ via the smash operation. This correspondence allows us to refine the expression for the star product of multisymmetric functions as follows:
\begin{thm}
Let $ e_{\alpha}(p) \) and \( e_{\beta}(q) $ be two multisymmetric functions. Then their star product is given by
$$
e_{\alpha}(p)\star e_{\beta}(q) = \sum_{m=0}^{M} \sum_{\Gamma} e_\Gamma(B(p,q)) \hbar^m,
\quad \text{where } \Gamma \text{ ranges over } \bigsqcup_{s=0}^{S} Q^{s},
$$
and the bound \( M \) is given by
$$
M = \max\left\{ \sum_{i=1}^{a} \sum_{j=1}^{b} S \gamma_{ij} : \gamma \in L(\alpha,\beta,n),\ (i,j)\in[a] \times [b] \right\},
$$
where $S$ is as defined in Proposition~\ref{Mmayor}.
\end{thm}

\begin{proof}
    Let $\Gamma \in Q(\alpha,\beta,n,m)$ be a solution that contributes a term to the star product. By Theorem~\ref{SMayor}, it follows that $\Gamma \in \bigsqcup_{s=0}^{S} Q^{s}$, and therefore,
    $$
    m=\sum_{i=0}^{a}\sum_{j=0}^{b}\sum_{k=0}^{\infty} k\Gamma_{ij}^{k}\leq
    \sum_{i=0}^{a}\sum_{j=0}^{b}\sum_{k=0}^{\infty} S\Gamma_{ij}^{k}=
    \sum_{i=0}^{a}\sum_{j=0}^{b} S\sum_{k=0}^{\infty}\Gamma_{ij}^{k}=
    \sum_{i=0}^{a}\sum_{j=0}^{b} S\gamma_{ij}.
    $$

    Taking the maximum over all $\gamma \in L(\alpha,\beta,n)$ yields
    $$
    m \leq M = \max\left\{\sum_{i=1}^{a}\sum_{j=1}^{b} S\gamma_{i,j} : \gamma \in L(\alpha,\beta,n) \text{ and } (i,j) \in [a] \times [b] \right\}.
    $$

    Consequently, the star product satisfies
    $$
    e_{\alpha}(p)\star e_{\beta}(q)=\sum_{m=0}^{M} \sum_{\Gamma} e_\Gamma(B(p,q)) \hbar^m \quad \text{where } \Gamma \in \bigsqcup_{s=0}^{S} Q^{s}.
    $$
\end{proof}

\section{Cubical liftings of matrices}

In the previous section, we established a precise connection between cubical matrices in $Q(\alpha, \beta, n, m)$ and matrices in $L(\alpha, \beta, n)$ via the smash operation. We now reverse this perspective by exploring how elements of $L(\alpha, \beta, n)$ can be lifted back into the cubical setting. This leads us to introduce the concept of \emph{lifting} a matrix to a cubical one at a given level. These liftings provide a constructive tool for generating and analyzing cubical matrices with prescribed smashes, which is crucial for understanding the structure and computation of the star product.

\begin{defi}
    Let $s \in \mathbb{N}$ and $m \in \mathbb{N}$. Given $\gamma \in L(\alpha, \beta, n)$, we define the \emph{lifting} of $\gamma$ at level $s$, denoted by $\uparrow_{m}^{s}(\gamma)$, as a cubical matrix $\Gamma \in Q^{s} \subseteq Q(\alpha, \beta, n, m)$ such that $\overline{\Gamma} = \gamma$.

  Furthermore, for any subset $B \subseteq L(\alpha, \beta, n)$, we define its lifting at level $s$ as
    $$
    \uparrow_{m}^{s}(B) = \left\{ \uparrow_{m}^{s}(\gamma) : \gamma \in B \right\}.
    $$
\end{defi}

\begin{exa}
    Consider $\alpha = (2, 1)$, $\beta = (1, 2)$, $n = 3$ and $m = 1$. Let

    $$\gamma = \begin{bmatrix}
        0 & 0 & 0\\
        0 & 0 & 2\\
        0 & 1 & 0
    \end{bmatrix},$$
    then the lifting of $\gamma$ to level 1, denoted by $\uparrow_{1}^{1}(\gamma)$, is given by:

    $$
    \uparrow_{1}^{1}(\gamma) = \left\{
    \left[
    \begin{bmatrix}
        0 & 0 & 0\\
        0 & 0 & 1\\
        0 & 1 & 0
    \end{bmatrix},
    \begin{bmatrix}
        0 & 0 & 0\\
        0 & 0 & 1\\
        0 & 0 & 0
    \end{bmatrix}
    \right],
    \left[
    \begin{bmatrix}
        0 & 0 & 0\\
        0 & 0 & 2\\
        0 & 0 & 0
    \end{bmatrix},
    \begin{bmatrix}
        0 & 0 & 0\\
        0 & 0 & 0\\
        0 & 1 & 0
    \end{bmatrix}
    \right]
    \right\}.
    $$
\end{exa}

Since $\overline{Q(\alpha,\beta,n,m)} = L(\alpha,\beta,n)$, lifting the entire set $L(\alpha,\beta,n)$ recovers the set $Q(\alpha,\beta,n,m)$. In other words,
$$
\displaystyle\bigsqcup_{s=0}^{m} \uparrow_{m}^{s}(L(\alpha,\beta,n)) = Q(\alpha,\beta,n,m),
$$
which enables us to express the star product of two elementary multisymmetric functions in terms of liftings of classical matrices $\gamma$.

For notational convenience, we write
$$
\uparrow_{m}(L(\alpha,\beta,n)) = \displaystyle\bigsqcup_{s=0}^{m} \uparrow_{m}^{s}(L(\alpha,\beta,n)).
$$

This reformulation is summarized in the following proposition.

\begin{thm}\label{PropestrellaFromL}
    The star product $\star$ between $e_\alpha(p)$ and $e_\beta(q)$, as given in Theorem~\ref{TeoImportante}, can be rewritten as
    $$
    e_\alpha(p)\star e_\beta(q)=\sum_{m=0}^{M}\left(\sum_{\Gamma\in \uparrow_{m}(L(\alpha,\beta,n))} e_\Gamma(B(p,q))\right)\hbar^m.
    $$
\end{thm}
\begin{rem}
    This formulation offers a computational advantage: rather than enumerating over the full set of cubical matrices $ Q(\alpha, \beta, n, m) $, it is enough to work with the classically defined matrices $ \gamma \in L(\alpha, \beta, n)$ and their corresponding liftings. This not only simplifies the structure of the objects involved, but also provides a more tractable framework for explicit calculations and algorithmic implementations of the star product.
\end{rem}

A particularly interesting aspect of the star product \( \star \) between multisymmetric functions is the role played by the \( \hbar \) coefficients. In fact, the coefficient of \( \hbar \) is directly related to the Poisson bracket between multisymmetric functions, specifically equal to half its value. For a more detailed explanation of this connection, we refer the reader to \cite{Kon}. Since an elementary multisymmetric function \( e_\Gamma \) appears as the coefficient of \( \hbar \) if and only if \( \Gamma \in Q(\alpha, \beta, n, 1) \), we turn our attention to its cardinality in order to determine how many such functions contribute at first order in the deformation.

\begin{prop} 
Let \( \alpha \in \mathbb{N}^a \), \( \beta \in \mathbb{N}^b \), and fix \( n \in \mathbb{N} \). Then the number of cubical matrices in \( Q(\alpha, \beta, n, 1) \) is given by
\[
|Q(\alpha, \beta, n, 1)| = \sum_{\gamma \in L(\alpha, \beta, n)} f_\gamma,
\]
where for each matrix \( \gamma \), the value \( f_\gamma \) counts the number of nonzero entries not located in the first row or first column, that is,
\[
f_\gamma = \left| \{ (i, j) \in [a] \times [b] : \gamma_{ij} \neq 0 \} \right|.
\]
\end{prop}

\begin{proof}
The proof follows directly from the identity
\[
Q(\alpha, \beta, n, 1)=\uparrow_{1}^{1}(L(\alpha,\beta,n))=\bigsqcup_{\gamma\in L(\alpha,\beta,n)}\uparrow_{1}^{1}(\gamma),
\]
which expresses that each matrix in $Q(\alpha, \beta, n, 1)$ arises as a level-$1$ lifting of some matrix $\gamma \in L(\alpha, \beta, n)$. Since the set $\uparrow_{1}^{1}(\gamma)$ contains exactly $f_\gamma$ elements—one for each nonzero entry in $\gamma$ (excluding row and column zero)—we conclude that
\[
|Q(\alpha, \beta, n, 1)| = \sum_{\gamma\in L(\alpha,\beta,n)} f_\gamma.
\]
\end{proof}
\section{The RSK algorithm in the quantum context}

The origins of the RSK algorithm trace back to the 1938 work of Gilbert de Beauregard Robinson \cite{Robinson}, who established a correspondence between permutations and pairs of standard Young tableaux. In 1961, Craige Schensted \cite{Schen} rediscovered and reformulated this correspondence as a bijection between permutations and pairs of standard tableaux, now known as the \emph{Schensted insertion algorithm}. Later, in 1970, Donald Knuth \cite{Knuth} extended this construction to handle matrices with nonnegative integer entries, giving rise to the now classical \emph{Robinson--Schensted--Knuth (RSK) algorithm}. This algorithm has since become a central tool in algebraic combinatorics, with deep connections to representation theory, symmetric functions, and algebraic geometry \cite{Sagan,Stanley}.

In this section, we adapt an idea originally introduced in the classical setting of \( L(\alpha, \beta, n) \), as discussed in \cite{PSierra}. There, a variant of the RSK algorithm was used to establish a bijection capable of generating each matrix \(\gamma\). In certain cases, this bijective approach proved simpler and more efficient than alternative methods.

Here, we extend this idea to the quantum setting, specifically to the structure of the star product. This extension provides a solid theoretical foundation for computational implementation \cite{sierra2025}, and underpins the Python code used to produce many of the examples in this work.

To this end, we develop a combinatorial framework that encodes cubical matrices \(\Gamma \in Q(\alpha, \beta, n, m)\) using a new family of objects, which we call \emph{3-words}. These objects are designed to mirror the structure of \(\Gamma\), and they allow us to construct a bijection between the elements of \(\uparrow_m(L_N)\) and a restricted class of 3-words. The results that follow demonstrate how a variant of the RSK algorithm can be employed to generate these elements efficiently, thereby making the computation of the star product more tractable.

\begin{defi}
Let $N \in \mathbb{N}$ be fixed. A \emph{3-word} $\Omega$ is a $3 \times N$ matrix with non-negative integer entries,
\[
\Omega = 
\begin{pmatrix}
s_1 & s_2 & \cdots & s_N \\
i_1 & i_2 & \cdots & i_N \\
j_1 & j_2 & \cdots & j_N
\end{pmatrix}
\in M_{3 \times N}(\mathbb{N}),
\]
satisfying the following conditions:
\begin{enumerate}
        \item $0 \leq s_1 \leq s_2 \leq \cdots \leq s_N$,
        \item For all $t \in [N]$, we have $i_t > 0$ and $j_t > 0$,
        \item If $s_t = s_{t+1}$, then $i_t \leq i_{t+1}$,
        \item If $i_t = i_{t+1}$, then $j_t \leq j_{t+1}$.
    \end{enumerate}
\end{defi}

\begin{defi}
    Let $\Omega$ be a 3-word and let $c \in \{1, 2, 3\}$. The \emph{type along row $c$}, denoted by $type^{c}(\Omega)$, is the vector
    $$
    type^{c}(\Omega) = (u_1, \dots, u_k),
    $$
    where each $u_i \in \mathbb{N}$ represents the number of times the integer $i$ appears in the $c$-th row of $\Omega$. The index $k$ corresponds to the largest value appearing in that row.
\end{defi}

\begin{exa}
Consider the 3-word
\[
\Omega =
\begin{pmatrix}
0 & 0 & 1 \\
1 & 2 & 3 \\
2 & 2 & 3
\end{pmatrix}.
\]
Then \( type^{1}(\Omega) = (1) \), \( type^{2}(\Omega) = (1,1,1) \), and \( type^{3}(\Omega) = (0,2,1) \).
\end{exa}

\begin{defi}
Let \( \alpha \in \mathbb{N}^a \), \( \beta \in \mathbb{N}^b \), \( n, m \in \mathbb{N} \). We define \( A(\alpha, \beta, n, m) \) to be the set of \( 3 \)-words
\[
\Omega =
\begin{pmatrix}
s_1 & s_2 & \cdots & s_N \\
i_1 & i_2 & \cdots & i_N \\
j_1 & j_2 & \cdots & j_N
\end{pmatrix}
\in M_{3 \times N}(\mathbb{N})
\]
satisfying the following conditions:
\begin{enumerate}
    \item For all \( t \in [N] \), if \( i_t = j_t \), then \( i_t \neq 1 \),
    \item If \( s_t > 0 \), then \( i_t > 1 \) and \( j_t > 1 \),
    \item \( \sum_{t = 1}^{N} s_t = m \),
    \item \( type^2(\Omega)  = (N - |\alpha|, \alpha_1, \dots, \alpha_a) \) and
          \( type^3(\Omega)  = (N - |\beta|, \beta_1, \dots, \beta_b) \).
\end{enumerate}
\end{defi}

\begin{thm}\label{thm:3words-bijection}
There exists a bijection between the set \( \uparrow_m(L_N) \) and the set \( A(\alpha, \beta, n, m) \).
\end{thm}

\begin{proof}
    We define the map
    $$
    \begin{array}{ccl}
        \psi: \uparrow_{m}(L_N) & \longrightarrow & A(\alpha, \beta, n, m)  \\
        \Gamma & \longmapsto & \psi(\Gamma) = (\varphi(\Gamma_{11}^1), \cdots, \varphi(\Gamma_{ab}^m)),
    \end{array}
    $$
    where each $\varphi(\Gamma_{ij}^k)$ is given by:

     $\varphi(\Gamma_{ij}^k):=\varphi(i,j,k,\Gamma_{ij}^{k})=\begin{array}{c}
		\underbrace{\left(\begin{matrix}
				k&k&\cdots&k\\
				i+1&i+1&\cdots&i+1\\
				j+1&j+1&\cdots&j+1\end{matrix}\right)}\\ \Gamma_{i,j}^k\  times\end{array}$.\\

    The resulting 3-word $\psi(\Gamma)$ is obtained by concatenating all such blocks.
    Since $\varphi$ is injective, it is clear that if $\Gamma, \Gamma' \in \uparrow_{m}(L_N)$ and $\psi(\Gamma) = \psi(\Gamma')$, then
  $$
\varphi(\Gamma_{ij}^k) = \varphi({\Gamma^{\prime}}_{ij}^k) \Rightarrow \Gamma_{ij}^k = {\Gamma^{\prime}}_{ij}^k,
$$
    for all $i,j,k$, and thus $\Gamma = \Gamma'$. Therefore, $\psi$ is injective.

    For surjectivity, let
    $$
    \Omega = \begin{pmatrix}
        s_1 & s_2 & \cdots & s_N \\
        i_1 & i_2 & \cdots & i_N \\
        j_1 & j_2 & \cdots & j_N
    \end{pmatrix} \in A(\alpha, \beta, n, m) .
    $$

    We now verify that $\Gamma \in \uparrow_{m}(L_N)$:
    \begin{enumerate}
        \item Since $\Omega \in A$ satisfies that if $i_t = j_t$ then $i_t \neq 1$, this ensures $\Gamma_{00}^k = 0$ for all $k$.
        \item If $s_t > 0$, then $i_t > 1$ and $j_t > 1$, so any pair $(i,j)$ with $i = 0$ or $j = 0$ must satisfy $\Gamma_{ij}^k = 0$ for $k \geq 1$.
        \item The total number of columns in $\Omega$ is $N \leq n$, so $|\Gamma| \leq n$, and        $$
        \sum_{k=0}^{\infty} \sum_{i=0}^a \sum_{j=0}^b k \Gamma_{ij}^k = \sum_{t=1}^{N} s_t = m.
        $$
        \item From the definition of type$^2$ and type$^3$, we conclude that:
        $$
        \sum_{j=0}^{b} \sum_{k=0}^{\infty} \Gamma_{ij}^k = \alpha_i, \quad
        \sum_{i=0}^{a} \sum_{k=0}^{\infty} \Gamma_{ij}^k = \beta_j.
        $$
    \end{enumerate}
 Therefore, $\Gamma \in \uparrow_{m}(L_N)$ and $\psi(\Gamma) = \Omega$, which completes the proof.
\end{proof}

\begin{exa}\label{3Palabra}
Let $\alpha = (2,1)$ and $\beta = (1,2)$. Consider $\Gamma \in \uparrow_{2}^{1}(L_3)$ given by
\[
\Gamma =
\left(
\begin{array}{cc}
\begin{bmatrix}
0 & 0 & 0 \\
0 & 0 & 0 \\
0 & 0 & 1
\end{bmatrix} &
\begin{bmatrix}
0 & 0 & 0 \\
0 & 1 & 1 \\
0 & 0 & 0
\end{bmatrix} \\
\Gamma^0 & \Gamma^1
\end{array}
\right).
\]
The associated $3$-word $\Omega$ is then
\[
\Omega =
\begin{pmatrix}
0 & 1 & 1 \\
3 & 2 & 2 \\
3 & 2 & 3
\end{pmatrix}.
\]
\end{exa}

These results highlight the usefulness of 3-words as a combinatorial tool for encoding and recovering the structure of elements \( \Gamma \in Q(\alpha, \beta, n, m) \). Given a 3-word associated to a cubical matrix \( \Gamma \), several important properties can be directly extracted: the total weight \( m \) corresponds to the sum \( \sum_{t=1}^{N} s_t \); the types of the second and third rows of the 3-word recover the data \( \alpha \) and \( \beta \) via \( \text{type}^2(\Omega) = (N - |\alpha|, \alpha_1, \dots, \alpha_a) \) and \( \text{type}^3(\Omega) = (N - |\beta|, \beta_1, \dots, \beta_b) \); the support level \( s \) of \( \Gamma \) is given by the maximum value \( s_N \); and the size \( |\Gamma| \) equals \( N \). Altogether, these features demonstrate the power of the bijection in simplifying both the description and manipulation of the combinatorial data involved in the star product. The following theorem summarizes these observations.

\begin{thm}
Let $\Gamma \in Q(\alpha, \beta, n, m)$ and consider its associated $3$-word $\Omega_{\Gamma}$ given by
\[
\Omega_{\Gamma} = \begin{pmatrix}
s_1 & s_2 & s_3 & \cdots & s_{N} \\
i_1 & i_2 & i_3 & \cdots & i_{N} \\
j_1 & j_2 & j_3 & \cdots & j_{N}
\end{pmatrix}.
\]
Then the following properties hold:
\begin{enumerate}
    \item $|\Gamma| = N$.
    \item $\Gamma \in Q^{s}$ where $s = s_{N}$.
    \item $m = \displaystyle\sum_{t=1}^{N} s_{t}$.
    \item The vector $\alpha$ is obtained by removing the first entry of $type^{2}(\Omega)$. Similarly, $\beta$ is obtained by removing the first entry of $type^{3}(\Omega)$.
\end{enumerate}
\end{thm}

\begin{proof}
\begin{enumerate}
    \item Recall the inverse mapping:
    \[
    \varphi^{-1}
    \begin{pmatrix}
    k & k & \cdots & k \\
    i & i & \cdots & i \\
    j & j & \cdots & j
    \end{pmatrix}
    = (i-1, j-1, k, \Gamma_{i-1, j-1}^{k}),
    \]
    where $\Gamma_{i-1, j-1}^{k}$ corresponds to the number of columns in the matrix above—that is, the number of times the triple $(k, i, j)$ appears in $\Omega$.

    To prove that $|\Gamma| = N$, we observe:
    \[
    |\Gamma| = \sum_{i=0}^{a} \sum_{j=0}^{b} \sum_{k=0}^{s_{N}} \Gamma_{ij}^{k}
    = \sum_{i=1}^{a+1} \sum_{j=1}^{b+1} \sum_{k=0}^{s_{N}} \pi_4 \circ \varphi^{-1}
    \begin{pmatrix}
        k & k & \cdots & k \\
        i & i & \cdots & i \\
        j & j & \cdots & j
    \end{pmatrix} = N,
    \]
    where $\pi_4$ denotes projection onto the fourth coordinate, returning the value $\Gamma_{i-1, j-1}^{k}$.

    \item To prove that $\Gamma \in Q^{s}$ for $s = s_{N}$, note that the sequence $s_1 \leq s_2 \leq \cdots \leq s_{N}$ implies that $s_{N}$ is the largest value in the first row of $\Omega_{\Gamma}$. Each $s_t$ corresponds to a nonzero entry $\Gamma_{ij}^{k}$ with $k = s_t$, so
    \[
    s_{N} = \max \{ k : \exists (i, j) \text{ such that } \Gamma_{ij}^{k} \neq 0 \}.
    \]
    Therefore, $\Gamma \in Q^{s_{N}}$.

    \item Using Theorem \ref{TeoImportante} and the identity $\sum_{t=1}^{\Gamma_{i-1, j-1}^{k}} s_{t} = k \Gamma_{i-1, j-1}^{k}$, we obtain:
    \[
    \sum_{t=1}^{N} s_t = \sum_{i=0}^{a} \sum_{j=0}^{b} \sum_{k=0}^{s_{N}} k \Gamma_{ij}^{k} = m.
    \]

    \item We know that $type^{2}(\Omega_{\Gamma}) = (u_{1}, u_{2}, \ldots, u_{k})$, where
    \[
    u_i = \sum_{j=1}^{b+1} \sum_{k=0}^{m} \Gamma_{i-1, j-1}^{k} = \sum_{j=0}^{b} \sum_{k=0}^{m} \Gamma_{i, j}^{k}.
    \]
    For $i > 1$, this yields $\sum_{j=0}^{b} \sum_{k=0}^{m} \Gamma_{i, j}^{k} = \alpha_i$, as required. The argument for $type^{3}(\Omega_\Gamma)$ is analogous.
\end{enumerate}
\end{proof}

\subsection*{Acknowledgments} 
The author wishes to thank Rafael Díaz and José Luis Ramírez for their valuable comments, insightful observations and constructive feedback. This work was supported by the Pontificia Universidad Javeriana, Bogotá, Colombia.
\bibliographystyle{plain} 
\bibliography{samplex.bib} 

\end{document}